\documentclass[11pt,a4paper]{article}

\usepackage{amsmath,amssymb,latexsym,graphics,epsfig}
\usepackage{hyperref}
\usepackage{color}
\usepackage{amsthm}
\usepackage{graphicx,url}
\usepackage{amsopn}
\usepackage{amssymb}
\usepackage{soul}

\newtheorem{Theorem}{Theorem}[section]
\newtheorem{Lemma}[Theorem]{Lemma}

\newtheorem{Proposition}[Theorem]{Proposition}
\newtheorem{Definition}[Theorem]{Definition}

\numberwithin{equation}{section}

\DeclareMathOperator{\modul}{mod}
\DeclareMathOperator{\tr}{tr}

\def\Z{\mathbb Z}

\def\e{\mbox{\boldmath $e$}}
\def\j{\mbox{\boldmath $j$}}
\def\n{\mbox{\boldmath $n$}}

\def\u{\mbox{\boldmath $u$}}
\def\x{\mbox{\boldmath $x$}}
\def\y{\mbox{\boldmath $y$}}

\def\0{\mbox{\boldmath $0$}}
\def\1{\mbox{\boldmath $1$}}

\def\A{\mbox{\boldmath $A$}}

\def\I{\mbox{\boldmath $I$}}

\def\O{\mbox{\boldmath $O$}}

\def\R{\mbox{\boldmath $R$}}

\def\Z{\mathbb Z}

\newcommand\blfootnote[1]{%
	\begingroup
	\renewcommand\thefootnote{}\footnote{#1}%
	\addtocounter{footnote}{-1}%
	\endgroup
}

\begin{document}

\title{On $d$-Fibonacci digraphs
		\thanks{This research has been partially supported by 
			AGAUR from the Catalan Government under project 2017SGR1087 and by MICINN from the Spanish Government under project PGC2018-095471-B-I00. The research of the first author has also been supported by MICINN from the Spanish Government under project MTM2017-83271-R.}}
			
	\author{
C. Dalf\'o\\
		{\small Departament de Matem\`atica, Universitat de Lleida} \\
		{\small Igualada (Barcelona), Catalonia} \\
				\vspace{.25cm}
		{\small {\tt{cristina.dalfo@matematica.udl.cat}}}\\
	M. A. Fiol\\
    	{\small Departament de Matem\`atiques, Universitat Polit\`ecnica de Catalunya} \\
    	{\small Barcelona Graduate School of Mathematics} \\
    	{\small Barcelona, Catalonia} \\
    	{\small {\tt{miguel.angel.fiol@upc.edu}}}
    }
	\date{}

	\date{}

\date{}

\maketitle

\blfootnote{
		\begin{minipage}[l]{0.3\textwidth} \includegraphics[trim=10cm 6cm 10cm 5cm,clip,scale=0.15]{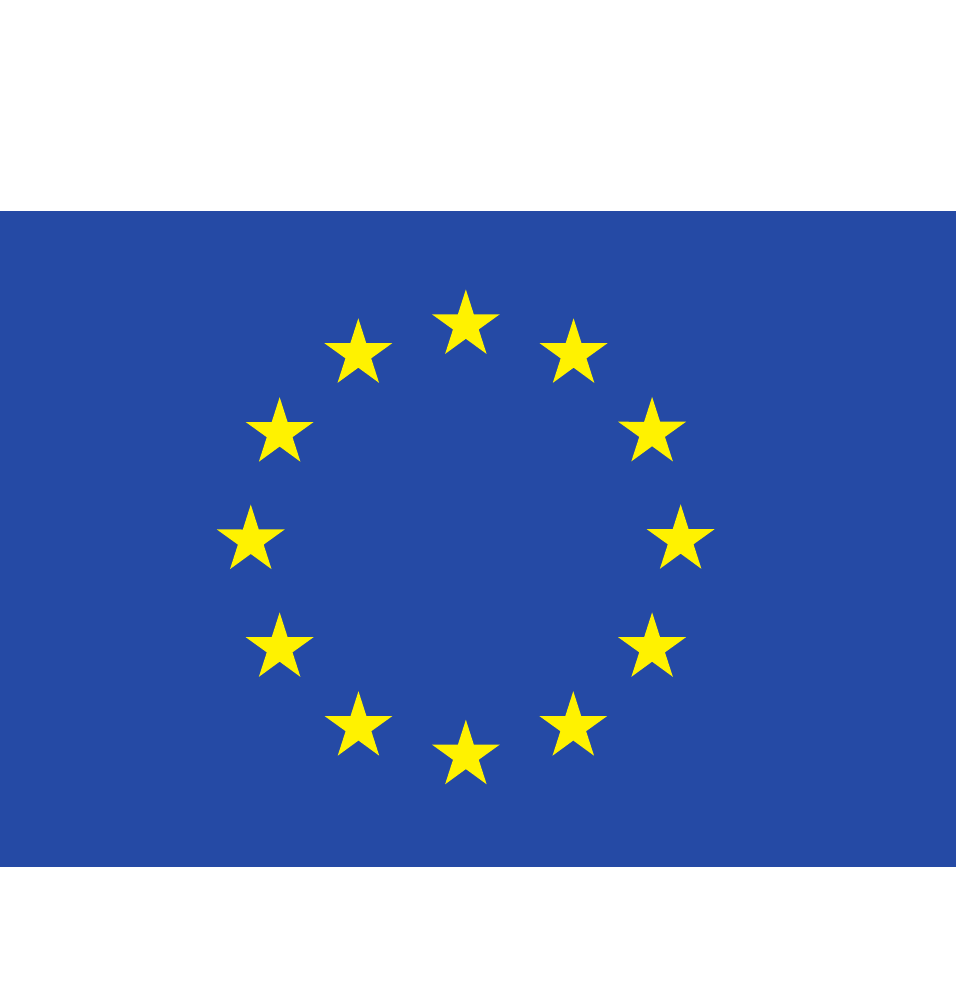} \end{minipage}  \hspace{-2cm} \begin{minipage}[l][1cm]{0.79\textwidth}
			The research of the first author has also received funding from the European Union's Horizon 2020 research and innovation programme under the Marie Sk\l{}odowska-Curie grant agreement No 734922.
	\end{minipage}}

\begin{abstract}
The $d$-Fibonacci digraphs $F(d,k)$, introduced here, have the number of vertices following generalized Fibonacci-like sequences.
They can be defined both as digraphs on alphabets and as iterated line digraphs.
Here we study some of their nice properties. For instance, $F(2,k)$ has diameter $d+k-2$ and is semi-pancyclic, that is,  it has a cycle of every length between 1 and $\ell$, with $\ell\in\{2k-2,2k-1\}$.
Moreover, it turns out that several other numbers of $F(d,k)$ (of closed $l$-walks, classes of vertices,
etc.) also follow the same linear recurrences as the numbers of vertices of the $d$-Fibonacci digraphs.
\end{abstract}




\section{Preliminaries}

Let us first introduce some basic notation and results.
A digraph $G=(V,E)$ consists of a (finite) set
$V=V(G)$ of vertices and a set $E=E(G)$ of arcs (directed edges) between vertices
of $G$. As the initial and final vertices of an arc are not necessarily different, the
digraphs may have \emph{loops} (arcs from a vertex to itself),
and \emph{multiple arcs}, that is, there can be more than one arc from each vertex
to any other. If $a=(u,v)$ is an arc from $u$ to  $v$, then vertex $u$ (and arc $a$)
is {\em adjacent to} vertex $v$, and vertex $v$ (and arc $a$) is {\em adjacent from}
$v$. The converse digraph $\overline{G}$ is obtained from $G$ by reversing the direction of each arc.
Let $G^+(v)$ and $G^-(v)$ denote the set of arcs adjacent from and to vertex
$v$, respectively.  A digraph $G$ is $k$\emph{-regular} if $|G^+(v)|=|G^-(v)|=k$ for all $v\in V$. As usual, we called \emph{cycle} to a closed walk in which all its vertices are different.

The adjacency matrix $\A$ of a digraph  $G=(V,E)$ is indexed by the vertices in $V$, and it has entries $(\A)_{uv}=\alpha$ if there are $\alpha$ arcs from $u$ to $v$, with $\alpha\ge 0$. Notice that, as we allow loops, the diagonal entries of $\A$ can be different from zero.

In the line digraph $LG$ of a digraph $G$, each vertex of $LG$ represents an
arc of $G$, that is, $V(LG)=\{uv : (u,v)\in E(G)\}$; and vertices $uv$ and $wz$ of $L(G)$ are adjacent if and only if $v=w$, namely, when the arc $(u,v)$ is adjacent to the arc $(w,z)$ in $G$.
The $k$-iterated line digraph $L^kG$ is recursively defined as $L^0G=G$ and $L^kG=L^{k-1}LG$ for $k\geq 1$.
It can easily be seen that every vertex of $L^kG$ corresponds to a walk
$v_0,v_1,\ldots,$ $v_k$ of length $k$ in $G$, where $(v_{i-1},v_{i})\in E$ for $i=1,\ldots,k$.
Then, if there is one arc between pairs of vertices and $\A$ is the adjacency matrix of $G$, the $uv$-entry of the power $\A^k$, denoted by $a_{uv}^{(k)}$, corresponds to the number of $k$-walks from vertex $u$ to vertex $v$ in $G$. The order $n_k$ of $L^kG$ turns out to be
\begin{equation}\label{orderL^kG}
n_k=\j\A^k\j^{\top},
\end{equation}
where $\j$ stands for the all-1 vector. If there are multiple arcs between pairs of vertices, then the corresponding entry in the matrix is not 1, but the number of these arcs.

If $G$ is a strongly connected  $d$-regular digraph, different from a directed cycle, with
diameter $D$, then its line digraph $L^kG$ is
$d$-regular with $n_k=d^kn$ vertices and has (asymptotically optimal) diameter $D+k$. In fact, for a strongly connected general digraph, the first author \cite{Da17} proved  that
the iterated line digraphs are always asymptotically dense. For more details, see Harary and Norman~\cite{HaNo60}, Aigner~\cite{Ai67}, and Fiol, Yebra, and Alegre~\cite{fya84}.
%

Given integers $d\ge 2$ and $k\ge 1$, the {\em de Bruijn digraph} $B(d,k)$ is commonly defined as a digraph on alphabet in the following way. This digraph has vertices $x_1x_2\ldots x_k$ with $x_{i}\in [0,d-1]$ for every $i=1,2,\ldots,k$.
Moreover, every vertex $x_1x_2\ldots x_k$ is adjacent to the vertices $x_2\ldots x_k x_{k+1}$, where $x_{k+1}\in [0,d-1]$.

For the concepts and results on digraphs not presented here, 
see, for instance, Bang-Jensen and Gutin~\cite{BJGu09},
Chartrand and Lesniak~\cite{cl96} or Diestel~\cite{d10}. 
%
%
%
\subsection{Generalized Fibonacci numbers}

A proposed generalization of the well-known Fibonacci numbers is the following.
Given an integer $d\ge 2$, the {\em $d$-step Fibonacci numbers} $F_1^{(d)},F_2^{(d)},F_3^{(d)},\ldots$ are defined through the linear recurrence relation
\begin{equation}
\label{recurF(d)}
F_k^{(d)}=\sum_{i=1}^{d}F_{k-i}^{(d)},
\end{equation}
initialized with $F_k^{(d)}=0$ for $k\le 0$ and $F_1^{(d)}=F_2^{(d)}=1$.
Thus, the cases $d=2,3,4,\ldots$ correspond to the so-called {\em Fibonacci numbers} $F_k$, {\em tribonacci numbers}, {\em tetrabonacci numbers}, etc., respectively. For more information, see, for example, Miles~\cite{Mi60}.

In particular, the Fibonacci numbers hold the recurrence $F_k=F_{k-1}+F_{k-2}$, which, as it is well known, is satisfied by  the numbers of the form
\begin{equation}
\label{recur(d=2)}
f(k)=a\phi^k+b\psi^k=a\left(\frac{1+\sqrt{5}}{2}\right)^k+b\left(\frac{1-\sqrt{5}}{2}\right)^k,
\end{equation}
where $a$ and $b$ are constants, $\phi=\frac{1+\sqrt{5}}{2}$ is the golden ratio, and $\psi=\phi^{-1}$. Recall also that, from $F_0=0$ and $F_1=1$, we get $a=-b=1/\sqrt{5}$, giving the Binet's formula  $F_k=(1/\sqrt{5})(\phi^k-\psi^k)$.

\section{$d$-Fibonacci digraphs on alphabets}
\label{sec:Falphabet}

\begin{Definition}
\label{def1}
For some given integers $d\ge 2$ and $k\ge 1$, the {\em $d$-Fibonacci digraph} $F(d,k)$ has vertices $\x=x_1x_2\ldots x_k$, where  for $i=0,\ldots,k-1$,  $x_{i+1}\in [0,d-1]$ if $x_{i-1}=0$, and  $x_{i+1}=x_i+1\ (\modul d)$ otherwise.
Moreover, every vertex $x_1x_2\ldots x_k$ is adjacent to the vertices $x_2\ldots x_k x_{k+1}$, where  $x_{k+1}\in [0,d-1]$ if $x_{k}=0$, and  $x_{k+1}=x_k+1\ (\modul d)$ otherwise.
\end{Definition}

For instance, the $2$-Fibonacci digraphs $F(2,k)$ with $k\le 4$ and $2,3,5,8$ vertices, are shown in Figure \ref{fig:F(2,k)}, whereas the 1-Fibonacci digraphs  $F(d,1)$ on $d$ vertices, with $d\in[2,5]$,  are depicted in Figure \ref{fig:T_d}.

\begin{figure}[t]
	\begin{center}
		\includegraphics[width=16cm]{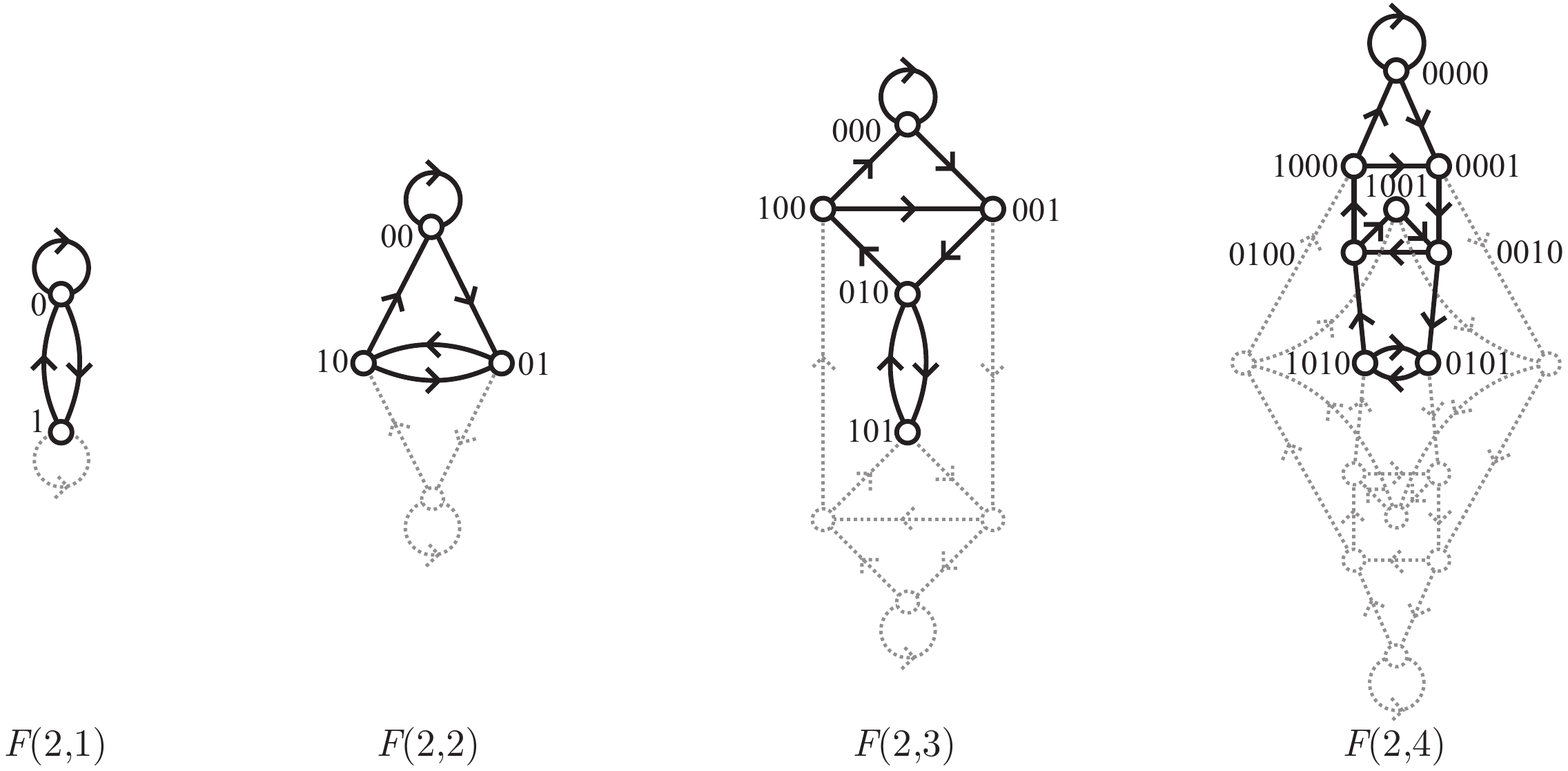}
	\end{center}
	\vskip-5.5cm
	\caption{The $2$-Fibonacci digraphs $F(2,k)$ for $k=1,2,3,4$ as subdigraphs of the de Bruijn digraphs.}
	\label{fig:F(2,k)}
\end{figure}

Some simple properties of the $d$-Fibonacci digraphs, which are easy consequences of their definition, are the following.

\begin{Lemma}
Let $F(d,k)$ be the $d$-Fibonacci digraph, with vertices $\x=x_1x_2\ldots x_k$, for $x_i\in [0,d-1]$.
\begin{itemize}
\item[$(i)$]
The out-degree of $\x$ is $\deg(\x)=d$ if $x_k=0$, and $\deg(\x)=1$ otherwise.
\item[$(ii)$]
The digraph $F(d,k)$ is an induced subdigraph of the de Bruijn digraph $B(d,k)$.
\item[$(iii)$]
The digraph $F(d,k)$ contains $F(d',k)$ as an induced subdigraph, for every $d'\leq d$.
\item[$(iv)$]
There is a homomorphism $\phi$ from $F(d,k)$ to $F(d,k')$, for every $k'\leq k$.
\item[$(v)$]
The automorphism group of $F(d,k)$ is the trivial one. Moreover, the digraph $F(2,k)$ is isomorphic to its converse.
\end{itemize}
\end{Lemma}

\begin{proof}
$(i)$ follows immediately from Definition \ref{def1}. Similarly, $(ii)$ is a direct consequence of the definitions of $F(d,k)$ and $B(d,k)$.
Concerning $(iii)$, notice that the vertices of $F(d,k)$ corresponding to the sequences $x_1x_2\ldots x_k$ with $x_i\in\{0,d-1,\ldots,d-d'+1\}$ induce a subdigraph isomorphic to $F(d',k)$. Alternatively, from the results of Section \ref{line-digraphs}, note that $T_{d'}$ is clearly an induced subdigraph of $T_d$ for every $d'\le d$ and, hence, the same property is inherited by $F(d',k)=L^kT_{d'}$ and $F(d,k)=L^kT_{d}$. To prove $(iv)$, we only need to exhibit the homorphism from $F(d,k)$ to $F(d,k')$, which is the following map on the corresponding sets of vertices
\begin{eqnarray*}
\phi : \ V(F(d,k)) & \rightarrow & V(F(d,k'))\\
  \ \x=x_1x_2\ldots x_k &  \rightarrow & \phi(\x)= x_{k-k'+1} x_{k-k'+2}\ldots x_k.
\end{eqnarray*}
Indeed, observe that if $\x\rightarrow \y$, then $\phi(\x)\rightarrow \phi(\y)$.
To prove the first part of $(v)$, we only need to realize that every automorphism of $F(d,k)$ must send the unique cycles of lengths 1 and 2 (loop and digon) to themselves. This means that vertex $00\ldots 0$ must be fixed, and the vertex set  $\{0101,\ldots, 1010\ldots\}$ must be an orbit. But the only way to preserve the adjacencies between these two vertices and $00\ldots 0$ is to fix them, which implies that all the other vertices have to be also fixed, and the automorphism is the identity.
Finally, the second statement of $(v)$ is justified by the mapping $x_1x_2\ldots x_k \mapsto x_k\ldots x_2 x_1$, which is an isomorphism between $F(2,k)$ and its converse $\overline{F(2,k)}$.
\end{proof}

In contrast with $F(2,k)$, the Fibonacci digraph $F(d,k)$ with $d>2$ is not isomorphic to its converse.
By using the line digraph approach of Section \ref{line-digraphs} again, this is a simple consequence of the fact that, for $d>2$, $T_d\not\cong \overline{T_d}$. However, the same approach allows to show that most of the properties of $F(d,k)$ related with $d$-Fibonacci numbers, are shared by its converse $\overline{F(d,k)}$.

To illustrate case $(ii)$, in  Figure \ref{fig:F(2,k)} each Fibonacci digraph $F(2,k)$ with $k\le 4$ is shown with thick lines  as a subdigraph of its corresponding de Bruijn digraph $B(d,k)$. In particular, note that
$F(d,1)$ has $d$ vertices, which coincides with the order $d^k$ of the de Bruijn digraph $B(d,k)$ when $k=1$.
In contrast, the number of vertices of $F(d,k)$ is much smaller when $k$ increases, as the following result shows.

\begin{Proposition}
	\label{prop:N}
The numbers of vertices $N(d,k)$ of the $d$-Fibonacci digraphs $F(d,k)$ satisfy the same linear recurrence as the $d$-step Fibonacci numbers  in \eqref{recurF(d)}
\begin{equation}
N(d,k+1)=\sum_{i=k-d+1}^k N(d,i),
\end{equation}
but now
initialized with $N(d,i)=1$ for $i=d-2,d-1,\ldots,0$, and $N(d,1)=d$.
\end{Proposition}

\begin{proof}
For $j\in [0,d-1]$, let $n_j^k$ be the number of vertices $x_1x_2\ldots x_k$ of $F(d,k)$ such that $x_k=j$. Thus, $n_j^1= 1$ for $j\in[0,d-1]$, $N(d,k)=\sum_{j=0}^{d-1}n_j^k$  and, from the conditions on the digits $x_i$, we get
\begin{align*}
n_0^{k+1} &= n_0^{k}+n_{d-1}^{k},\\
n_1^{k+1} & = n_0^{k},\\
 n_2^{k+1}&= n_0^{k}+n_{1}^{k},\\
 & \vdots \\
 n_{d-1}^{k+1} &=n_0^{k}+n_{d-2}^{k},
\end{align*}
 or, in matrix form,
\begin{align}
\n^{k+1} &=\left(n_0^{k+1},n_1^{k+1},n_2^{k+1},\ldots,n_{d-1}^{k+1}\right) \nonumber\\
 &=\left(n_0^{k},n_1^{k},n_2^{k},\ldots,n_{d-1}^{k}\right)\left(\begin{array}{cccccc}
1 & 1 & 1 & 1 & \ldots & 1\\
0 & 0 & 1 & 0 & \ldots & 0\\
0 & 0 & 0 & 1 & \ldots & 0\\
\vdots  & \vdots  & \vdots  & \vdots  & \ddots & \vdots\\
0 & 0 & 0 & 0 & \ldots & 1\\
1 & 0 & 0 & 0 & \ldots & 0
\end{array}\right):=\n^{k}\R.
\label{R}
\end{align}
Then, applying recursively \eqref{R}, $\n^{k+1}=\n^{1}\R^k=\j\R^k$.
Now, it is readily checked that the characteristic polynomial of the above recurrence $d\times d$  matrix $\R$ is $\phi(x)=x^{d}-\sum_{i=0}^{d-1} x^i$. Indeed, since
$$
\phi(x)=\det (\R-x\I)=\det \left(\begin{array}{cccccc}
x-1 & -1 & -1 & -1 & \ldots & -1\\
0 & x & -1 & 0 & \ldots & 0\\
0 & 0 & x & -1 & \ldots & 0\\
\vdots  & \vdots  & \vdots  & \vdots  & \ddots & \vdots\\
0 & 0 & 0 & 0 & \ldots & -1\\
-1 & 0 & 0 & 0 & \ldots & x
\end{array}\right),
$$
we can expand the determinant relative to the first line to get $\phi(x)=(x-1)x^{d-1}-x^{d-2}-x^{d-3}-\cdots-1$ (notice that, in doing so, every $(d-1)\times (d-1)$ submatrix have only one transversal with nonzero product $(x-1)x^{d-1}$, $-x^{d-2}$, $x^{d-3}$, etc.)
Then, from $\R^{k-d}\phi(\R)=\O$, where $\O$ is the all-0 matrix, we get
\begin{equation}
\label{recurR}
\R^{k}=\sum_{i=k-d}^{k-1} \R^i
\end{equation}
or, multiplying both terms by the vector $\n^{1}=\j$,
\begin{equation}
\label{recur-vectors}
\j\R^{k}=\n^{k+1}=\sum_{i=k-d}^{k-1} \j\R^i=\sum_{i=k-d}^{k-1} \n^{i+1}=\sum_{i=k-d+1}^{k} \n^{i}.
\end{equation}
Hence,
$$
N(d,k+1)=\sum_{j=0}^{d-1}n_j^{k+1}=\sum_{j=0}^{d-1}\sum_{i=k-d+1}^k n_j^{i}=\sum_{i=k-d+1}^k \sum_{j=0}^{d-1} n_j^{i} =\sum_{i=k-d+1}^k N(d,i),
$$
as claimed.
Besides, to show that the recurrence can be initialized with the $d$ values  $N(d,i)=1$ for $i=d-2,d-1,\ldots,0$ and $N(d,1)=d$, we need  to show that $N(d,2)=N(d,1)+ d-1$,
$N(d,3)=N(d,2)+N(d,1)+d-2$, \ldots, $N(d,d)=\sum_{i=1}^{d-1}N(d,i)+1$.
With this aim, note first that $N(d,k)=\sum_{j=0}^{d-1} n_j^k=\n^{k}\j^{\top}=\j\R^{k-1}\j^{\top}$ for $k=1,\ldots,d$, so that we first compute the vectors $\u^{k-1}=\R^{k-1}\j^{\top}$ for $k=0,1,\ldots,k+1$ to get:
\begin{align*}
\u^0 &=\j^{\top}=(1,1,1,\stackrel{(d)}{\ldots},1,1)^{\top},\\
\u^1 &=\R\j^{\top}=(d,1,1,\stackrel{(d-1)}{\ldots},1,1)^{\top},\\
\u^2 &=\R^2\j^{\top}=\R(\u^1)^{\top}=(d+(d-1),1,\stackrel{(d-2)}{\ldots},1,d)^{\top},\\
\u^3 &=\R^3\j^{\top}=\R(\u^2)^{\top}=(2d+(d-1)+(d-2),1,\stackrel{(d-3)}{\ldots},1,d,d+(d-1))^{\top},\\
 & \vdots \\
\u^{d} &=\R^{d}\j^{\top}=\R(\u^{d-1})^{\top}\\
  &=(2^{d-2}d+2^{d-3}(d-1)+\cdots+1,d,d+(d-1),\ldots,2^{d-3}d+2^{d-4}(d-1)+\cdots+2)^{\top}.
\end{align*}
Notice that, for each $k=1,\ldots,d$, the sum of all entries of $\u^{k-1}$ equals the first entry $u^k_0$ of $\u^k$. Consequently,
$N(d,k)=\j(\u^{k-1})^{\top}=u_k^0$, so giving
\begin{align*}
N(d,2) &=u^2_0=d+(d-1)=N(d,1)+d-1\\
N(d,3) &=u^3_0=2d+(d-1)+d-2=N(d,2)+N(d,1)+d-2\\
  &\vdots\\
N(d,d) &=u^d_0=2^{d-2}d+2^{d-3}(d-1)+\cdots+1=\sum_{i=1}^{d-1}N(d,i)+1,
\end{align*}
as required.
\end{proof}

Notice that, from \eqref{recur-vectors}, we proved that not only the total number of vertices of $F(d,k)$, but also those vertices whose sequences end with a given digit $j\in[0,d-1]$ satisfy the same recurrence
 as the $d$-step Fibonacci numbers  in \eqref{recurF(d)}. For example, for $d=5$, Table \ref{tab:vectorsF(5,k)} shows the vectors $\n^k$, for $k=0,\ldots,7$, with entries being such number of sequences. Then, we can observe the claimed recurrence $n_j^{k+1}=n_j^{k}+\cdots +n_j^{k-4}$, for $k\ge 5$, by looking at each $j$-th column of the formed array.

\begin{table}[t]
\begin{center}
\begin{tabular}{|c |c c c c c|}
\hline
 $j$ &  $0$ & $1$ & $2$ &  $3$ & $4$  \\
\hline
 $\n^0$ & 1 & 1 & 1 & 1 & 1 \\
 $\n^1$ & 2 & 1 & 2 & 2 & 2 \\
 $\n^2$ & 4 & 2 & 3 & 4 & 4 \\
 $\n^3$ & 8 & 4 & 6 & 7 & 8 \\
 $\n^4$ & 16 & 8 &  12 & 14 & 15 \\
$\n^5$ & 31 & 16 & 24 & 28 & 30 \\
 $\n^6$ & 61 & 31 & 47 & 55 & 59 \\
 $\n^7$ & 120 & 61 & 92 & 108 & 116 \\
\hline
\end{tabular}
\caption{The vectors $\n^k$, for $k=0,\ldots,7$, with entries $n_j^k$ being the numbers of vertices $x_1x_2\ldots x_k$ of $F(5,k)$ such that $x_k=j\in[0,4]$.}
\label{tab:vectorsF(5,k)}
\end{center}
\end{table}

\section{Fibonacci digraphs}

Although a similar (although more involved) study for general $d$ can be done, we concentrate here in the case $d=2$, where we simply refer to Fibonacci digraphs $F(k)$. The reason is that, from Proposition \ref{prop:N}, the numbers $N(k)=N(2,k)$ of vertices
 of the (2-)Fibonacci digraphs $F(k)=F(2,k)$ are
$$
N(1)=2,\  N(2)=3,\  N(3)=5,\  N(4)=8,\  N(5)=13,\  N(6)=21,\ldots
$$
which corresponds to the standard Fibonacci sequence $F_3,F_4,F_5,F_6,F_7,F_8\ldots$, see again Figure \ref{fig:F(2,k)}.
Indeed, it is known that
the number of binary sequences of length $k$ without consecutive 1's is the Fibonacci number $F_{k+2}$. For example, among the $16$ binary sequences of length $k=4$, there are $F_6 = 8$ without consecutive 1's. Namely, $0000$, $0001$, $0010$, $0100$, $0101$, $1000$, $1001$, and $1010$, which are the vertices of $F(2,4)$ in Figure \ref{fig:F(2,k)}. 
Indeed, such binary sequences also correspond to the vertices of the (undirected) \emph{Fibonacci graphs} that are induced subgraphs of the $k$-cubes. So, two vertices are adjacent when their labels differ exactly in one digit. In Figure \ref{fig:sub-hipercubs}, there are represented the four first Fibonacci graphs. For more information, see  Hsu, Page, and Liu~\cite{HsPaLi93}.

\begin{figure}[t]
	\begin{center}
		\includegraphics[width=16cm]{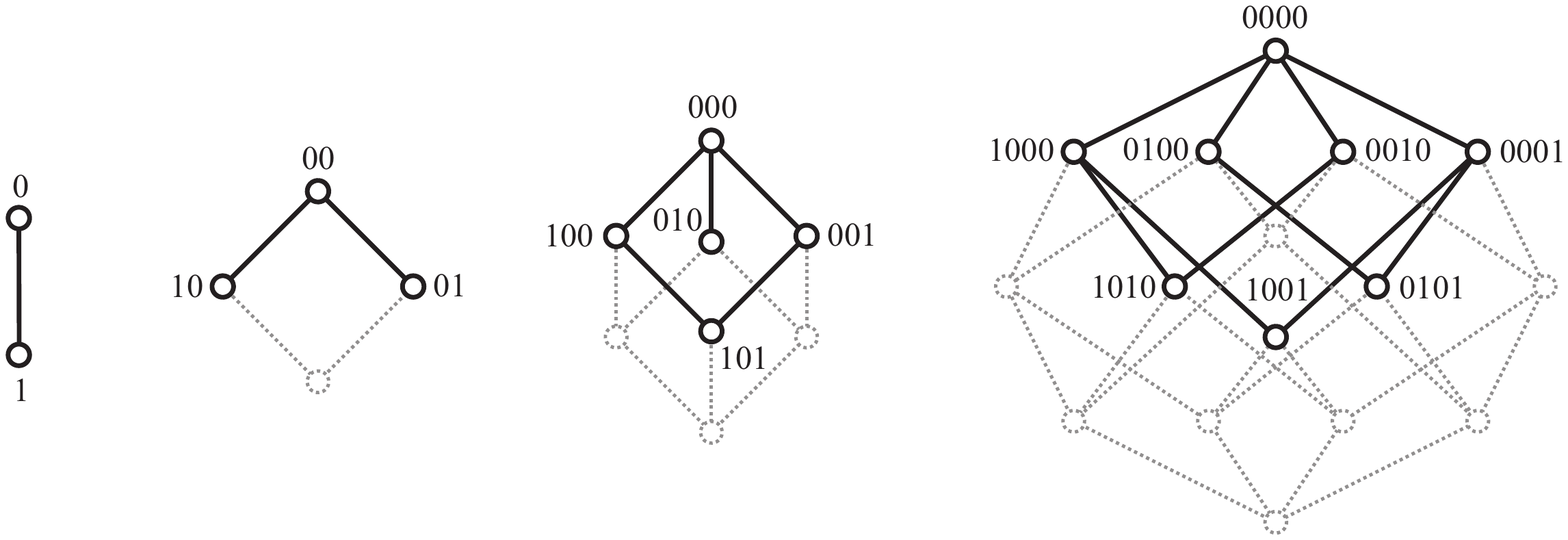}
	\end{center}
	\vskip-7.5cm
	\caption{The four first Fibonacci graphs as subgraphs of the hypercubes.}
	\label{fig:sub-hipercubs}
\end{figure}

Now, let us show a result on the lengths of the cycles in the Fibonacci digraphs. More precisely, we prove that $F(k)$ is semi-pancyclic.

\begin{Proposition}
For every $k\ge 2$, let  $\ell=2k-2$ if $k$ is odd and $\ell=2k-1$ if $k$ is even. Then, the Fibonacci digraph $F(k)$ is $(1,\ell)$-pancyclic, that is, it contains a cycle of every length $1,2,\ldots, \ell$.
\end{Proposition}

\begin{proof}
A {\em $p(>1)$-periodic} vertex of $F(k)$ has $1$'s in the positions $i(\le k)$, $i+p$, $i+2p$, \ldots Then, by cyclically shifting at the left the corresponding sequence, but keeping the periodicity, such a vertex gives rise to a cycle of length $p$. For instance in $F(7)$, vertex
$0001000$ gives the $4$-cycle
$$
0001000\rightarrow 0010001\rightarrow 0100010\rightarrow 1000100\rightarrow 0001000.
$$
The other cycles (of lengths $k+1,k+2,\ldots,\ell$) go either through the vertex $\0=00\ldots0$ or the vertex $\1=00\ldots01$. In both cases, if we look at  the successive sequences of the cycle as the rows of an array, the entries 1 form a number  $q=1,2,\ldots,\lfloor k/2\rfloor$ of anti-diagonals, as shown in Table \ref{tab:7-10} for $k=7$ and $q=1,2,3$. We label the corresponding cycles with the prefixes $[\0,q]$ and $[\1,q]$, respectively. Then, summarizing, we have the following cases:
\begin{itemize}
\item
The vertex $\0$ gives a cycle of length 1 (a loop).
\item
The $p$-periodic vertices give cycles of length $p$ for $p=2,3,\ldots,k$.
\item
The $[\1,q]$-cycles, containing vertex $\1$,  have length $k+2q-2$ for $q=1,2,\ldots,\lfloor k/2\rfloor$. (For $q=1$, the cycle of length $k$ is also obtained in the previous case for $p=k$.)
\item
The $[\0,q]$-cycles, containing vertex $\0$, have lengths $k+2q-1$ for $q=1,2,\ldots,\lfloor k/2\rfloor$.
In particular, for $q=\lfloor k/2\rfloor$, the $[\0,\lfloor k/2\rfloor]$-cycle has length $\ell=k+2\lfloor k/2\rfloor-1\in\{2k-2,2k-1\}$, as required.
\end{itemize}
This completes the proof.
\end{proof}

\begin{table}[t]
\begin{center}
\begin{tabular}{|c@{}c@{}c@{}c@{}c@{}c@{}c|}
\hline
 $x_1$ &  $x_2$ & $x_3$ & $x_4$ &  $x_5$ & $x_6$ & $x_7$ \\
\hline
 0 & 0 & 0 & 0 & 0 & 0 & 0 \\
 0 & 0 & 0 & 0 & 0 & 0 & 1 \\
 0 & 0 & 0 & 0 & 0 &  1 & 0\\
 0 & 0 & 0 & 0 &  1 & 0 &  $0$ \\
 0 & 0 & 0 &  1 & 0 & $0$ & 0\\
0 & 0 & 1 & 0 & $0$ & 0 & $0$\\
 0 &  1 & 0 & $0$ & 0 & $0$ & 0 \\
 1 & 0 & $0$ & 0 & $0$ & 0 & 0\\
\hline
\end{tabular}
\qquad
\begin{tabular}{|c@{}c@{}c@{}c@{}c@{}c@{}c|}
\hline
 $x_1$ &  $x_2$ & $x_3$ & $x_4$ &  $x_5$ & $x_6$ & $x_7$ \\
\hline
 0 & 0 & 0 & 0 & 0 & 0 & 0 \\
 0 & 0 & 0 & 0 & 0 & 0 & 1 \\
 0 & 0 & 0 & 0 & 0 &  1 & 0\\
 0 & 0 & 0 & 0 &  1 & 0 &  1 \\
 0 & 0 & 0 &  1 & 0 & 1 & 0\\
 0 & 0 & 1 & 0 & 1 & 0 & 0\\
 0 &  1 & 0 & 1 & 0 & 0 & 0 \\
1 & 0 & 1 & 0 & 0 & 0 & 0\\
 0 & 1 & 0 & 0 & 0 & 0 & 0\\
 1 & 0 & 0 & 0 & 0 & 0 & 0 \\
\hline
\end{tabular}
\qquad
\begin{tabular}{|c@{}c@{}c@{}c@{}c@{}c@{}c|}
	\hline
	$x_1$ &  $x_2$ & $x_3$ & $x_4$ &  $x_5$ & $x_6$ & $x_7$ \\
	\hline
	0 & 0 & 0 & 0 & 0 & 0 & 0 \\
	0 & 0 & 0 & 0 & 0 & 0 & 1 \\
	0 & 0 & 0 & 0 & 0 &  1 & 0\\
	0 & 0 & 0 & 0 &  1 & 0 &  1 \\
	0 & 0 & 0 &  1 & 0 & 1 & 0\\
	0 & 0 & 1 & 0 & 1 & 0 & 1\\
	0 &  1 & 0 & 1 & 0 & 1 & 0 \\
	1 & 0 & 1 & 0 & 1 & 0 & 0\\
	0 & 1 & 0 & 1 & 0 & 0 & 0\\
	1 & 0 & 1 & 0 & 0 & 0 & 0 \\
	0 & 1 & 0 & 0 & 0 & 0 & 0\\
	1 & 0 & 0 & 0 & 0 & 0 & 0\\
	\hline
\end{tabular}
\caption{Cycles of lengths 7-8, 9-10, and 11-12 in $F(7)$.}
\label{tab:7-10}
\end{center}
\end{table}

\section{$d$-Fibonacci digraphs as iterated line digraphs}
\label{line-digraphs}

The following result shows that the $d$-Fibonacci digraphs can also be constructed as iterated line digraphs.
Let $T_d$ be the digraph with set of vertices $\Z_d$ and arcs $(0,i)$ for every $i\in \Z_d$, and arcs $(i,i+1)$ for every $i=\Z_d\setminus 0$.
Thus, $T_d$ has $d$ vertices and $2d-1$ arcs. Moreover, it is a strongly connected digraph with diameter $D=d-1$.
As examples, see Figure \ref{fig:T_d}.

The adjacency matrix $\A$ of $T_d$, indexed by the vertices $0,1,\ldots,d-1$, has first row $\j$, the all-1 vector, and $i$-th row the unit vector $\e_{i+1}$, for $i=1,2,\ldots,d-1$ (recall that the arithmetic is modulo $d$). Then, $\A$ coincides with the recurrence matrix $\R$ in \eqref{R} and, hence,
the entries of the powers of $\A$ satisfy the recurrence
\begin{equation}
\label{recurA^{k+1}}
(\A^{k+1})_{uv}=\sum_{i=k-d+1}^k (\A^{i})_{uv}
\end{equation}
for $k\ge d$.

In the following result, we show that the $d$-Fibonacci digraphs can also be defined as iterated line digraphs of $T_d$.

\begin{Proposition}
	\label{prop:F=L^kT}
The $d$-Fibonacci digraph $F(d,k)$ coincides with the $(k-1)$-iterated line digraph of $T_d$, that is $F(d,k)=L^{k-1}T_d$, for $k\geq0$, with  $F(d,1)=L^0T_d=T_d$.
\end{Proposition}

\begin{proof}
We know that the vertices of $L^{k-1}T_d$ correspond to the walks of length ${k-1}$ in $T_d$. But, according to Definition \ref{def1}, such walks are in correspondence with the sequences of length $k$ defining the vertices of $F(d,k)$. Moreover, the adjacencies in  $L^{k-1}T_d$ are the same as in $F(d,k)$.
\end{proof}

\begin{figure}[t]
	\begin{center}
		\includegraphics[width=16cm]{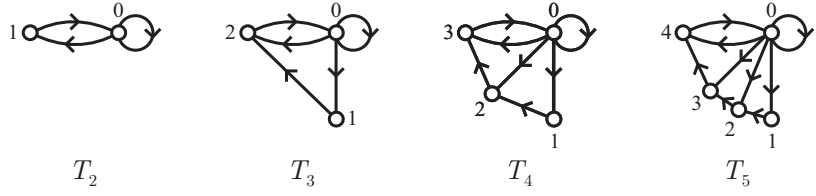}
	\end{center}
	\vskip-10cm
	\caption{The digraphs $T_d=F(d,1)$ for $d=2,3,4,5$.}
	\label{fig:T_d}
\end{figure}

As a consequence of the last proposition and the proof of Proposition \ref{prop:N},
we have the following result.

\begin{Proposition}
\label{propo2}
Let $F(d,k)$ be the $d$-Fibonacci digraph with $N=N(d,k)$ vertices given by Proposition  \ref{prop:N}.  Let $\A(k)$ and  $\A$ be, respectively, the adjacency matrices of  $F(d,k)$ and  $T_d$. 
\begin{itemize}
\item[$(i)$]
The diameter of  the $d$-Fibonacci digraph $F(d,k)$ is $D=k+d-1$.
\item[$(ii)$]
The eigenvalues of  $F(d,k)$ are the $d$ zeros of the polynomial $p(x)=x^d-x^{d-1}-x^{d-2}-\cdots-1$ (or, alternatively, the $d$ zeros different from 1 of the polynomial $q(x)=x^{d+1}-2x^d+1$) plus $N-d$ zeros.
\item[$(iii)$]
Fo any given $d,k\ge 0$, the total number of closed $l$-walks $C_l(d,k)$ in $F(d,k)$ satisfies the same linear recurrence as the $d$-step Fibonacci numbers  in \eqref{recurF(d)}, initiated with $C_l(d,k)=\tr \A^l$ for $l=0,\ldots,d-1$.
\end{itemize}
\end{Proposition}

\begin{proof}
\begin{itemize}
	\item[$(i)$]
Since the diameter of $T_d$ is  $D=d-1$ , the result follows from Proposition \ref{prop:F=L^kT} and the results in
Fiol, Yebra, and Alegre~\cite{fya84}.
\item[$(ii)$]
Since $\A=\R$, the characteristic polynomial $\phi(x)$ of $T_d$  is
$\phi_d(x)=x^d-\sum_{r=0}^{d-1} x^r$. Hence, from the results in Balbuena, Ferrero, Marcote, and Pelayo~ \cite{BaFeMaPe03}, the characteristic polynomial of $F(d,k)=L^k T_d$ is
$\psi(x)=x^{N-d}\phi(x)$, which gives the result.
\item[$(iii)$]
From $(ii)$, the nonzero eigenvalues of $F(d,k)$ and $T(d)$ coincide. Then, for  $k\geq 1$, the total numbers of closed walks of length $l$, with $l\geq1$, in $F(d,k)$ and in $T_d$ coincide because $\tr \A(k)^l=\tr \A^l$. But $\A$ coincides with the recurrence matrix $\R$ in \eqref{R}, so that, from \eqref{recurA^{k+1}} and $l\ge d$,
\begin{equation}
\label{recur-closed-walks}
C_{l+1}(d,k)=\tr \A^{l+1}=\sum_{i=l-d+1}^l \tr \A^i = \sum_{i=l-d+1}^l C_i(d,k),
\end{equation}
and $C_l(d,k)=\tr \A^k$ for $l=0,\ldots,d-1$.
\end{itemize}
\end{proof}

For instance, in the case of Fibonacci digraphs ($d=2$), \eqref{recur-closed-walks} becomes the version of
\ref{recur(d=2)} for the number of closed walks in $F(2,k)$. Namely,
\begin{equation}
\label{recur-closed-w(d=2)}
C_l(2,k)=\phi^l+\psi^l=\left(\frac{1+\sqrt{5}}{2}\right)^l+\left(\frac{1-\sqrt{5}}{2}\right)^l,
\end{equation}
initiated with $C_0(2,k)=2$ and $C_1(2,k)=1$. Compare \eqref{recur-closed-w(d=2)} with the Binet's formula $F_l=(1/\sqrt{5})(\phi^l-\psi^l)$.

In fact, from \eqref{recurA^{k+1}} and the fact that every closed walk of length $l$ in $T_d$ gives a
closed walk of the same length in $F(d,k)$, and vice versa, we can prove that, for any given $j,d\in [0,d-1]$, the numbers $C_j(d,k)$ of closed walks in the digraphs $F(d,k)$ for $k\ge d$ follow the same recurrence of the $d$-step Fibonacci numbers.
What is more, the same holds for the total number of walks in $F(d,k)$, which go from the vertices of type
$x_1x_2\ldots j$ to the vertices of type $x_1x_2\ldots j'$ for any given $j,j'\in [0,d-1]$. In the case of $d=2$, this is a consequence of the following known formula for the powers of the adjacency matrix of $T_2$, as a particular case of \eqref{recurR},
$$
\A^k=\left(
\begin{array}{cc}
1 & 1\\
1 & 0
\end{array}
\right)^k
=
\left(
\begin{array}{cc}
1 & 1\\
1 & 0
\end{array}
\right)^{k-1}+
\left(
\begin{array}{cc}
1 & 1\\
1 & 0
\end{array}
\right)^{k-2}
=\left(
\begin{array}{cc}
F_{k+1} & F_k\\
F_k & F_{k-1}
\end{array}
\right)
$$
for $k\ge 2$.

The fact that $(\A^k)_{00}=F_k$ corresponds to the number of closed walks of length $k$ rooted at vertex $\0$ in the digraph $T_2$ is cited in On-line Encyclopedia of Integer Sequences A000045 \cite{Sl}.
%


\end{document}